\pdfpkmode={lexmarkr}
\pdfmapfile{}

\documentclass[a4paper,11pt,reqno]{amsart} 
\usepackage{mathtools}
\usepackage{amsfonts}
\usepackage[T1]{fontenc}
\usepackage[utf8]{inputenc}
\usepackage[mathscr]{euscript}
\usepackage{verbatim}
\usepackage{amssymb,latexsym}
\usepackage{amsthm}
\usepackage{eufrak}
\usepackage[shortlabels]{enumitem}
\usepackage{color}
\DeclareMathOperator{\curl}{curl}

\newtheorem{thm}{Theorem}[section]

\newtheorem{lem}[thm]{Lemma}

\newtheorem{lemma}[thm]{Lemma}
\newtheorem{proposition}[thm]{Proposition}
\newtheorem{corollary}[thm]{Corollary}

\theoremstyle{remark}

\newtheorem{rem}[thm]{Remark}

\newcommand{\Int}[2]{\displaystyle{\int_{#1}^{#2}}}

\newcommand{\Ab}{\mathbf {A}}

\newcommand{\R}{\mathbb{R}}

\def\sig#1{\vbox{\hsize=5.5cm
\kern2cm\hrule\kern1ex
\hbox to \hsize{\strut\hfil #1 \hfil}}}

\newcommand\signatures[4]{%
\vspace{3cm}
\hbox to \hsize{\hfil #1, \today\hfil}
\vspace{3cm}
\hbox to \hsize{\quad#2\hfil\hfil #3\quad}
\vspace{3cm}
\hbox to \hsize{\hfil#4\hfil}}
\numberwithin{equation}{section}

\title[Magnetic Robin Laplacian]{Counterexample to strong diamagnetism\\ for the magnetic Robin Laplacian}

\author[A. Kachmar]{Ayman Kachmar}
\address[Ayman Kachmar]{Lebanese University, Department of Mathematics, Nabatieh, Lebanon.}
\email{ayman.kashmar@gmail.com}
\author[M. P.-Sundqvist]{Mikael P. Sundqvist}
\address[Mikael Persson Sundqvist]{Lund University, Department of Mathematical 
Sciences, Box 118, 221\ 00 Lund, Sweden.}
\email{mikael.persson\_sundqvist@math.lth.se}

\begin{document}
\maketitle
\begin{abstract}
We determine a counterexample to strong diamagnetism for the Laplace operator in the unit disc with a uniform magnetic field and Robin boundary condition. The example follows from the accurate asymptotics of the lowest eigenvalue when the Robin parameter tends to $-\infty$.
\end{abstract}

\section{Introduction}

\subsection{Magnetic Robin Laplacian} 

We denote by $\Omega=\{x\in\R^2~:~|x|<1\}$ the open unit disk and by $\Gamma=\partial\Omega=\{x\in\R^2~:~|x|=1\}$ its boundary. 
We study the lowest eigenvalue of the magnetic Robin
Laplacian in $L^2(\Omega)$,
\begin{equation}\label{Shr-op-Gen}
\mathcal{P}^b_\gamma=-(\nabla-ib\Ab_0)^2,
\end{equation}
with domain
\begin{equation}\label{eq:bc}
D(\mathcal P^b_\gamma)=\{u\in H^2(\Omega)~:~\nu\cdot(\nabla u-ib\Ab_0)u+\gamma\,u=0\quad{\rm on}~\partial\Omega\}\,.
\end{equation}
Here $\nu$ is the unit outward normal vector of $\Gamma$, $\gamma<0$ the \emph{Robin} parameter and $b>0$  is the \emph{intensity} of the applied magnetic field. The vector field $\Ab_0$ generates the unit magnetic field and is defined as follows
\begin{equation}\label{eq:A0}
\Ab_0(x_1,x_2)=\frac12(-x_2,x_1)\,.
\end{equation}

To be more precise, the operator $\mathcal P^b_\gamma$ is defined as the Friedrichs extension,
starting from the 
quadratic form \cite[Ch.~4]{Hel},  
\begin{equation}\label{QF-Gen}
H^1(\Omega)\ni u\mapsto  \mathcal{Q}^b_\gamma(u):=\int_\Omega\bigl|(\nabla-ib\Ab_0) u(x)\bigr|^2\,dx+\gamma\Int{\Gamma}{}|u(x)|^{2}\,ds(x)\,.
\end{equation}
\subsection{Main result}

The operator $\mathcal{P}^b_\gamma$ has a compact resolvent, and thus its
spectrum consists of an increasing sequence of eigenvalues.
We are interested in examining the  asymptotics of the principal eigenvalue 
\begin{equation}\label{eq:p-ev}
\lambda_1(b,\gamma)=\inf_{u\in H^1(\Omega)}\frac{\mathcal{Q}^b_\gamma(u)}{\|u\|_{L^2(\Omega)}^2}
\end{equation} 
when $b>0$ is \emph{fixed} and the Robin parameter $\gamma$ tends to $-\infty$. 

\begin{thm}\label{thm:KS}
Let $b>0$. Then, as $\gamma\to-\infty$,
\[\lambda_1(b,\gamma)=-\gamma^2+\gamma +\inf_{m\in\mathbb Z}\Bigl(m-\frac{b}{2}\Bigr)^2
{-\frac{1}{2}}+o(1).
\]
\end{thm}

The first two terms in the asymptotic expansion given in Theorem~\ref{thm:KS} are well known after many contributions (see \cite{Pan, PP1, PP2} for the case $b=0$ and \cite{K-dia} for the case $b>0$);  however, the third correction term is new for the disc geometry for $b>0$. The recent contribution \cite[Thm.~1.5]{KOP} shows that Theorem~\ref{thm:KS} continues to hold in the case $b=0$.

\subsection{Lack of strong diamagnetism}

The celebrated diamagnetic inequality yields 
\[
\lambda_1(b,\gamma)\geq \lambda(0,\gamma)\,.
\]
By using  Theorem~\ref{thm:KS}, we can quantify the diamagnetic inequality  as follows
\begin{equation}\label{eq:dm}
\lambda_1(b,\gamma)-\lambda_1(0,\gamma)\underset{\gamma\to-\infty}{\sim} e(b):= \inf_{m\in\mathbb Z}\Bigl(m-\frac{b}{2}\Bigr)^2
\,.
\end{equation}
Connected to the diamagnetic inequality is the property of  strong diamagnetism \cite{E}; this is whether the function $b\mapsto\lambda_1(b,\gamma)$ is monotone increasing on some interval $[b_0,+\infty)\subset\R_+$.

As consequence of Theorem~\ref{thm:KS}, we obtain a counterexample to strong diamagnetism.
 
\begin{corollary}\label{corol:KS}
There exists $\gamma_0<0$ such that, for all $\gamma\in(-\infty,\gamma_0]$, the function $b\mapsto \lambda_1(b,\gamma)$ is not monotone increasing.
\end{corollary}

Besides its mathematical interest, the question of strong diamagnetism has applications to Physics, particularly in the context of superconductivity \cite{FH-b}. In the case of a simply connected domain subject to a uniform applied magnetic field and Neumann boundary condition ($\gamma=0$), strong diamagnetism holds  \cite{FH-cmp, FH-aif}. Counter examples of strong diamagnetism   exist for uniform magnetic fields in non-simply connected domains, or for non-uniform magnetic fields in simply connected domains \cite{FPS, HK-ring}. Interestingly, the Robin boundary condition has the unique feature where strong diamagnetism \emph{fails} for  the disc (which is a simply connected domain) even when it is subject to a \emph{uniform} applied magnetic field. 

Corollary~\ref{corol:KS} results from the following statement. 
Given a positive real number $A$, there exist $\gamma_0<0$ and $A<b_1<b_2<b_3$ 
such that, for all $\gamma\in(-\infty,\gamma_0]$,
\[
\lambda_1(b_1,\gamma)<\lambda_1(b_2,\gamma)~\&~\lambda_1(b_2,\gamma)>\lambda_1(b_{3},\gamma)\,.
\]
We can simply select the constants $b_i$ as follows
\[
b_{1}=2n_0\,,\quad b_2=2n_0+1\,,\quad b_{3}=2n_0+\frac32\,,
\]
where $n_0$ is the smallest natural number satisfying $n_0> A$; the conclusion then follows from Theorem~\ref{thm:KS}.

Using the periodicity of the function $b\mapsto e(b)$, given a  natural number     $N$, we can select  $\gamma_1<0$ such that
\[
b_{1,i}:=b_1+i<b_{2,i}:=b_2+i<b_{3,i}:=b_3+i
\]
with the following two inequalities
\[
\lambda_1(b_{1,i},\gamma)<\lambda_1(b_{2,i},\gamma)\,,\quad \lambda_1(b_{2,i},\gamma)>\lambda_1(b_{3,i},\gamma)\,,
\]
holding for all $\gamma\leq \gamma_1$ and $i\in\{1,2,\cdots,N\}$.  

\begin{figure}\label{fig:graphs}
\centering
\includegraphics{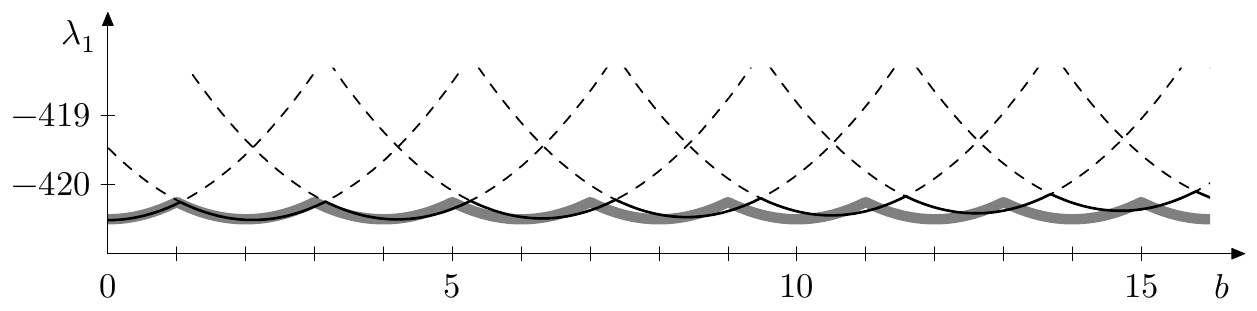}
\caption{We fix $\gamma=-20$. 
The solid black curve is the graph of the function $b\mapsto \lambda_1(b,\gamma)$ 
for $0<b<16$, calculated numerically using Wolfram Mathematica and the fact that
the eigenvalues of the fiber operators (see Section~\ref{sec:fiber}) 
satisfy certain equations involving Whittaker functions.
The dashed curves show how these eigenvalues
continue outside the interval where they give the minimum energy.
The gray thick curve in the background is the graph of
the function  $b\mapsto -\gamma^2+\gamma+\inf_{m\in\mathbb Z}(m-b/2)^2-1/2$.}
\end{figure}

\subsection{The Little--Parks effect}
When cooled below a certain \emph{critical temperature}, a normal conductor becomes a superconductor and looses electrical resistance.
The  Little--Parks experiment displays oscillations in the critical temperature of a superconductor as the applied magnetic field varies. Typically, the superconducting sample used in the experiment is a thin ring. From a mathematical perspective,   the  trivial \emph{normal}  solution of the Ginzburg--Landau equations changes back and forth from  stable to  unstable states.

Using a model for a superconductor with  enhanced surface \cite{FJ, MI}, we can use  Theorem~\ref{thm:KS} to estimate the critical temperature as a function of the applied magnetic field, consistent with the Little--Parks experiment. The novelty in our situation is that, unlike the Little--Parks experiment, the superconducting sample is a disc subject to a uniform magnetic  field.

The model we study is a variant of the Ginzburg--Landau energy by adding a (negative) surface energy term (amounting for the enhanced surface). 
The surface term can be derived naturally  starting  from a Ginzburg--Landau model for two adjust superconductors \cite[Thm.~1.2]{K-cocv}.

Following the presentation in~\cite[Sec. 3]{GS}, we introduce the  functional 
\begin{multline}\label{eq:GL}
\mathcal E^{\mathrm{phys}}(u,\mathbf a)=\frac{\hbar^2}{2m \ell}\int_{\partial\widetilde\Omega}|u|^2\,d\sigma(\tilde x)+\\
\int_{\widetilde\Omega}
\left(\frac1{2m}\left|\left(\hbar\nabla-i\frac{2e}{c}\mathbf a\right)u\right|^2+\alpha(T)|u|^2+\frac{\beta}{2}|u|^4+\frac1{8\pi}|\curl\mathbf a-H|^2\right)\,d\tilde x\,.
\end{multline} 
Here $(u,\mathbf a)$ describes the superconducting properties ($(u,\mathbf a)\equiv(0,H\Ab_0)$ signifies  the normal state, where $\Ab_0$ is introduced in \eqref{eq:A0}); $T$ denotes the temperature;  $\hbar$, $e$, $c$, $m$, $\beta$ 
are positive  constants; $H\geq 0$ measures the intensity  of the applied magnetic field   and $\ell<0$ models the enhanced surface. The disc $\widetilde\Omega=\{\tilde x\in\R^2~:~|\tilde x|<R\}$ is  the horizontal   cross section of the superconducting sample. The magnetic permeability $\mu_0$ in $\widetilde\Omega$ is assumed uniform, so we take $\mu_0=1$. For later use, we introduce the Ginzburg--Landau parameter
\begin{equation}\label{eq:kappa}
\kappa=\sqrt{\frac{m^2\beta c^2}{8\pi e^2\hbar}}\,.
\end{equation}
Among all the parameters in \eqref{eq:GL}, only $\alpha(T)$ depends on the temperature $T$; consequently, $\kappa$ is temperature independent. The expression of $\alpha(T)$ is given via the following relation
\begin{equation}\label{eq:c-l}
\alpha(T)=\frac{\hbar^2}{2m\xi_0^2}\left(\frac{T}{T_{c_0}}-1\right)\,,
\end{equation}
where $\xi_0>0$ is a temperature independent parameter, the  \emph{coherence length}  at zero temperature. The parameter $T_{c_0}$ is the critical temperature of the superconductor occupying $\widetilde\Omega$ in the absence of a magnetic field (i.e. when $H=0$).

We will express the functional in~\eqref{eq:GL} in temperature independent units,
and introduce some notation
\begin{align*}
\tilde x&=R x, & \widetilde\Omega&=R\Omega, &u(\tilde x) &= \sqrt{\frac{|\alpha(0)|}{\beta}}\psi(Rx),\\
\mathbf a(\tilde x)&=\frac{c\hbar}{2e\xi_0^2}R\Ab(Rx),&b&=\frac{c\hbar }{2e\xi_0^2} RH,& \gamma&=\frac{R}{\ell}.
\end{align*}
The functional in \eqref{eq:GL} becomes
\[
\mathcal E^{\mathrm{phys}}(u,\mathbf a)=\frac{|\alpha(0)|\hbar^2}{2m\beta}\mathcal E(\psi,\Ab)\,,
\]
where
\begin{multline}\label{eq:GL*}
\mathcal E(\psi,\Ab)=\gamma\int_{\partial\Omega}|\psi|^2\,d\sigma(x)+\\
\int_\Omega \biggl(|(\nabla-i\Ab)\psi|^2-\frac{R^2}{\xi_0^2}\mu(T)|\psi|^2+\frac{R^2}{2\xi_0^2}|\psi|^4+\kappa^2|\curl\Ab-b|^2\biggr)\,dx\,,
\end{multline}
and
\begin{equation}\label{eq:mu(T)}
\mu(T):=1-\frac{T}{T_{c_0}}\,.
\end{equation}
The functional in \eqref{eq:GL*} is defined on the space
\[
\mathcal H= H^1(\Omega;\mathbb C)\times H^1(\Omega;\R^2)\,.
\]
Clearly, the normal solution $(0,b\Ab_0)$ is a critical point of the functional in \eqref{eq:GL*}; it is said to be stable if it is a local minimizer. Using the direct method of the calculus of variations, we can prove that a minimizer $(\psi_*,\Ab_*)\in\mathcal H$ of $\mathcal E$ exists (cf. \cite[Sec.~3]{GS}).

Recall the eigevalue $\lambda(b,\gamma)$ introduced in \eqref{eq:p-ev}. Linearizing  the functional in \eqref{eq:GL*} near the normal state $(0,b\Ab_0)$, we get the following:
\begin{itemize}[--]
\item If $\lambda(b,\gamma)<\frac{R^2}{\xi_0^2}\mu(T)$, then the normal state is not stable and the global minimizer $(\psi_*,\Ab_*)$ is non-trivial in the sense that $\psi_*\not\equiv0$\,;
\item If $\lambda(b,\gamma)>\frac{R^2}{\xi_0^2}\mu(T)$, then the normal state is a local minimizer\,.
\end{itemize}
Consequently, we introduce the critical temperature $T_{c}(b)$, in the non-zero magnetic field $b$, as the solution of the equation
\[
\lambda(b,\gamma)=\frac{R^2}{\xi_0^2}\mu(T)\,.
\]
Thanks to~\eqref{eq:mu(T)} 
we find that
\begin{equation}\label{eq:Tc}
T_{c}(b)=\Bigl(1-\frac{\xi_0^2}{R^{2}}\lambda(b,\gamma)\Bigr)T_{c_0}\,.
\end{equation}
Using Theorem~\ref{thm:KS}, we can estimate $T_{c}(b)$ as $\gamma\to-\infty$; we find
\begin{equation}\label{eq:Tc*}
T_c(b)=\Bigl[1 -\frac{\xi_0^2}{R^{2}}\Bigl(-\gamma^2+\gamma+e(b)-\frac12\Bigr)\Bigr]T_{c_0}+o(1)\,,
\end{equation} 
where $e(b)$ is introduced in \eqref{eq:dm}. It is worth noticing that
\begin{itemize}[--]
\item $T_c(b)>T_{c_0}$\,;
\item Up to approximation errors, $T_c(b)$ is a periodic function of $b$, which is consistent with the Little--Parks effect\,;
\item For $T<T_c(b)$, the global minimizer of $\mathcal E$ is non-trivial (in the sense $\psi_*\not\equiv0$)\,; while for $T>T_c(b)$, the normal solution is a local minimizer of $\mathcal E$.
\end{itemize}

\subsection*{Acknowledgments} This work started when A. Kachmar visited the mathematics department in Lund University. The research of A. Kachmar is supported by a grant from  the Lebanese university within the project ``Analytical and numerical aspects of the Ginzburg--Landau
model''.

\section{Proof of Theorem~\ref{thm:KS}}

\subsection{Outline}

The proof consists of several reductions to operators that are easier
to handle. In the first step we change parameter to have a semi-classical
parameter. We then observe that we have localization close to the boundary,
make a Fourier decomposition, and express the interesting operators and quadratic forms
in suitable coordinates. Some effective operators appear, and we
expand their eigenvalues in terms of the semi-classical parameter.

Since we are not looking at the large magnetic field limit, the terms in the
potential that appears in polar coordinates is easier to handle since the
angular momentum and magnetic field strength do not compete against each other.

\subsection{Translation of Theorem~\ref{thm:KS} into a semi-classical statement}

It is convenient to work in a \emph{semi-classical} framework. We do so by 
introducing the semi-classical parameter $h=\gamma^{-2}$.
Then $h\to 0_+$ when $\gamma\to-\infty$, and the quadratic form $\mathcal Q^b_\gamma$
can be written as
\[
\mathcal Q^b_\gamma(u)
=
h^{-2}\biggl(\int_\Omega|(h\nabla-ibh\Ab_0)u(x)|^2\,dx-h^{3/2}\int_{\partial\Omega}|u(x)|^2\,ds(x)\biggr)\,.
\]
Consequently, we get the $h$-dependent self-adjoint operator
\begin{equation}\label{eq:Lh} 
\mathcal L_{h}^b=-(h\nabla-ibh\Ab_0)^2\,,
\end{equation} 
with domain
\begin{equation}\label{eq:Lh-dom}
D(\mathcal L_{h}^b)=\{u\in H^2(\Omega)~:~\nu\cdot (\nabla-ib\Ab_0)
u-h^{1/2}u=0{\rm ~on~}\partial\Omega\}\,.
\end{equation} 
The spectra of the operators $\mathcal P^b_\gamma$ and $\mathcal L_h^b$ are
related as
\[
\sigma (\mathcal P^b_\gamma)=h^{-2}\sigma(\mathcal L_h^b)\,.
\]
Let $\mu_1(h,b)$ be the principal eigenvalue of the operator  $\mathcal L_{h}^b$.
Theorem~\ref{thm:KS} can be rephrased as follows.

\begin{thm}\label{thm:KS'}
Let $b>0$. Then, as $h\to 0_+$,
\[
\mu_1(h,b)=-h-h^{3/2}+\left(\inf_{m\in\mathbb Z}\Bigl(m-\frac{b}{2}\Bigr)-\frac12\right)h^2+o(h^2)\,.
\]
\end{thm}

\subsection{Reduction to a thin ring}\label{sec:toring}
Our aim is to work in (a variant of) polar coordinates. However, if we change
directly to polar coordinates we will get some illusive problems at the origin
with negative powers of $r=|x|$. 

For small $h$, the ground states of the operator $\mathcal L_h^b$ are localized near 
the boundary of $\Omega$ (Proposition~\ref{prop:dec} below). This will allow
us to work in an annulus instead of the disk.
Before we give the localization result we show that there exists a sufficiently
small eigenvalue for small $h$.

\begin{lemma}\label{lem:existeig}
Let $b>0$. Then there exists $h_0\in(0,1)$ such that for all $h\in(0,h_0)$
\[
\mu_1(h,b)\leq -h-\frac12 h^{3/2}.
\]
\end{lemma}

\begin{proof}
Let $u(x)=c\exp(h^{-1/2}(|x|-1))$, where $c$ is chosen so that $u$ becomes normalized in $L^2(\Omega)$.
A direct calculation gives
\[
\frac{h^2\mathcal Q^b_\gamma(u)}{\|u\|^2}=-h-h^{3/2}+\mathcal O(h^2).
\]
By changing the coefficient in front of $h^{3/2}$ we get the existence of $h_0\in(0,1)$
such that the claimed inequality holds for $h\in(0,h_0)$.
\end{proof}

\begin{proposition}[localization of ground states]\label{prop:dec} Let $M\in(-1,0)$. For all $\alpha <
\sqrt{-M}$,  there exist constants $C>0$  and $h_0\in(0,1)$ such
that, if $u_h$ is a normalized ground state of $\mathcal L_h$ with eigenvalue
bounded above by $Mh$, then, for all $h\in(0,h_0)$,
\[
\int_\Omega \left(|u_h(x)|^2+h|(\nabla-ib\Ab_0) u_h(x)|^2\right)\exp\left(\frac{2\alpha\, {\rm
dist}(x,\partial\Omega)}{h^{1/2}}\right)\,dx\leq C\,.
\]
\end{proposition}

The proof of Proposition~\ref{prop:dec} is similar to the one of \cite[Thm.~5.1]{HK-tams}, 
and we leave out the details.

As a consequence of the concentration properties of the ground states, we can approximate the principal eigenvalue $\mu_1(h,b)$ by a ground state energy $\tilde\mu(h,b,\rho)$ that we describe next. 

Let $\rho\in(0,\frac12)$ and consider the annulus $\Omega_h=\{x\in\mathbb R^2~:~1-h^\rho<|x|<1\}$. We introduce the quadratic form
\begin{equation}\label{eq:ringF}
q_h^{b,\rho}(u)=h^2 \int_{\Omega_h} |(\nabla-ib\Ab_0)u(x)|^2\,dx-h^{3/2}\int_{|x|=1}|u(x)|^2\,ds\,
\end{equation}
defined on functions in $H^1(\Omega_h)$ with trace zero on the inner part of
the boundary $\Gamma_{\rm in}:=\{x\in\mathbb R^2~:~|x|=1-h^\rho$\}. This quadratic form is
related to a self-adjoint operator with mixed boundary conditions.
Its lowest eigenvalue $\tilde\mu_1(h,b,\rho)$ is given by
\begin{equation}\label{eq:mu1-ring}
\tilde\mu_1(h,b,\rho)
=
\inf\frac{q_h^{b,\rho}(u)}{\int_{\Omega_h}|u(x)|^2\,dx}
\end{equation}
where the infimum is taken over all $u$ in the domain of the quadratic form (i.e. $u\in H^1(\Omega_h)$ with $u=0$ on $\Gamma_{\rm in}$).

\begin{lemma} Assume that $b>0$ and $\rho\in(0,\frac12)$. Then there exists
$h_0\in(0,1)$ such that, for all $h\in(0,h_0)$,
\begin{equation}\label{eq:ring}
\mu_1(h,b)=\tilde \mu_1(h,b,\rho)+\mathcal O\Big(\exp\bigl(-h^{\rho-\frac12} \bigr)\Big).
\end{equation}
\end{lemma}

\begin{proof}
The inequality $\mu_1(h,b)\leq \tilde\mu_1(h,b,\rho)$ is not asymptotic. If $\tilde u$
is a function minimizing the quotient in \eqref{eq:mu1-ring}, then we can extend it by zero
inside the annulus. Inserting the new function into the quadratic form for $\mathcal L_h^b$, we
find that $\mu_1(h,b)\leq \tilde\mu_1(h,b,\rho)$.

To get a bound in the opposite direction, we cut off (smoothly) the eigenfunction
corresponding to $\mu_1(h,b)$, since it does not satisfy the correct boundary condition
if $|x|=1-h^\rho$. Thanks to Proposition~\ref{prop:dec} (with the choice $\alpha=\frac12$)
the error introduced is exponentially small.
\end{proof}

In light of \eqref{eq:ring}, we finish the proof of Theorem~\ref{thm:KS'} once we prove that
\begin{equation}\label{eq:ring*}
\tilde\mu(h,b,\rho)=-h-h^{3/2}+\left(\inf_{m\in\mathbb Z}\Bigl(m-\frac{b}{2}\Bigr)^2-\frac12\right)h^2+o(h^2)\,.
\end{equation}
The $m$ in the right-hand side stands for the quantized angular momentum. Our
next step is to make a Fourier expansion that will reduce our study to the
study of an infinite family (parametrized by $m\in\mathbb Z$) of ordinary
differential operators.

\subsection{Reduction to fiber operators}\label{sec:fiber}
We recall that $b>0$ and $\rho\in(0,\frac12)$ are considered to
be fixed constants.
In polar coordinates ($x_1=r\cos\theta$, $x_2=r\sin\theta$) the quadratic 
form $q_h^{b,\rho}$ reads
\[
h^2\biggl(\int_0^{2\pi}\int_{1-h^\rho}^1 \Bigl(|\partial_ru|^2+\frac{1}{r^2}\Bigl|\Bigl(\partial_\theta-i\frac{b}{2}r^2\Bigr)u\Bigr|^2\Bigr) r\,dr\,d\theta-h^{-1/2}\int_0^{2\pi}|u|^2\,d\theta\biggr).
\]
Next, we  use the completeness of the orthonormal family $\{e^{im\theta}/\sqrt{2\pi}\}_{m\in\mathbb Z}$
in $L^2([0,2\pi])$, and write 
\[
u(r,\theta)=\sum_{m\in\mathbb Z} u_m(r)\frac{e^{im\theta}}{\sqrt{2\pi}}.
\]
Here we assume that each $u_m$ belongs to $L^2((1-h^\rho,1),r\,dr)$. 
We are led to study the family of quadratic forms
\[
u_m\mapsto h^2\biggl(\int_{1-h^\rho}^1 \Bigl(|u_m'(r)|^2+\frac{1}{r^2}\Bigl|\Bigl(m-\frac{b}{2}r^2\Bigr)u_m(r)\Bigr|^2\Bigr) r\,dr-h^{-1/2}|u_m(1)|^2\biggr).
\]
Since we have localization to the outer circle, it is convenient to work with the
variable $\tau=h^{-1/2}(1-r)$, the scaled distance from $|x|=1$. We write 
$\tilde u_m(\tau)=u_m(r)$
and denote by 
\begin{equation}\label{eq:delta}
\delta:=h^{\rho-\frac12}\,,
\end{equation}
the upper limit of $\tau$. The relevant quadratic forms to study is
\begin{multline*}
\tilde u_m\mapsto 
\\
\int_{0}^\delta \Bigl\{|\tilde u_m'(\tau)|^2
+h(1-h^{1/2}\tau)^{-2}\Bigl|\Bigl(m-\frac{b}{2}(1-h^{1/2}\tau)^2\Bigr)\tilde u_m\Bigr|^2\Bigr\} (1-h^{1/2}\tau)\,d\tau\\
-
|\tilde u_m(0)|^2
\end{multline*}
The differential operator that corresponds to this quadratic form acts as
\begin{equation}\label{eq:H0b}
\mathcal H_{m,h}^{b,\rho}=-\frac{d^2}{d\tau^2}
+\frac{h^{1/2}}{1-h^{1/2}\tau}\frac{d}{d\tau}+\frac{h}{(1-h^{1/2}\tau)^{2}}\Bigl(m-\frac{b}{2}(1-h^{1/2}\tau)^2\Bigr)^2.
\end{equation}
With domain
\begin{equation}\label{eq:domH0b}
D(\mathcal H_{\beta,h})=\{u\in H^2((0,\delta))~:~u'(0)=-u(0)\quad{\rm and}\quad u(\delta)=0\}.
\end{equation}
$\mathcal H_{m,h}^{b,\rho}$ becomes self-adjoint in the weighted space 
$L^2((0,\delta),(1-h^{1/2}\tau)\,d\tau)$. 
We denote the smallest eigenvalue of $\mathcal H_{m,h}^{b,\rho}$ by
$\lambda_1(\mathcal H_{m,h}^{b,\rho})$.
From the completeness and orthogonality of the family $\{e^{im\theta}\}_{m\in\mathbb Z}$
it follows that
\begin{equation}\label{eq:ev=min}
\tilde\mu(h,b,\rho)=h\inf_{m\in\mathbb Z} \lambda_1\bigl(\mathcal H_{m,h}^{b,\rho}\bigr)\,.
\end{equation}
To take advantage of this equality we need information 
about $\lambda_1\bigl(\mathcal H_{m,h}^{b,\rho}\bigr)$. We will get the information
needed by comparing with simpler operators. In fact, we will first compare with
the weighted Laplace obtained by ignoring the third term in the right-hand side
of~\eqref{eq:H0b}. To do this, we first look at the simpler operator obtained by
ignoring also the second term.

\subsection{A 1D Laplacian}\
The spectrum of the operator
$-\frac{d^2}{d\tau^2}$ in $L^2(\R_+)$
with domain
$\{u\in H^2(\R_+)~:\,u'(0)=- u(0)\}$
is explicitly known (see~\cite{HK-tams}). It consists of
the simple eigenvalue $-1$ together with the interval $[0,+\infty)$.
A normalized eigenfunction corresponding to the eigenvalue $-1$ is given by
\begin{equation}\label{eq:u0}
u_0(\tau)=\sqrt{2}\exp(-\tau)\,.
\end{equation}

\subsection{A weighted 1D Laplacian}
Let $\rho\in(\frac14,\frac12)$ be a fixed constant. In the sequel, the parameter $h\in(0,1)$ varies so that
$h^{\frac12-\rho}<\frac13$.
We recall that $\delta=h^{\rho-\frac12}$ and note that  $\delta\to+\infty$ when $h\to0_+$.

In the weighted space $L^2\big((0,\delta);(1-h^{1/2}\tau)\,d\tau\big)$, we introduce  the self-adjoint
operator,
\begin{equation}\label{eq:Hh}
\mathcal H_{h}=-\frac{d^2}{d\tau^2}+
\frac{h^{1/2}}{1-h^{1/2}\tau}\frac{d}{d\tau}\,,\end{equation} with domain
\begin{equation}\label{eq:domHh}
D(\mathcal H_{h})=\{u\in H^2((0,\delta))~:~u'(0)=-u(0)\quad{\rm and}\quad u(\delta)=0\}\,.
\end{equation}
The operator $\mathcal H_{h}$ is defined starting from  the closed quadratic form
\[
q_{h}(u)=\int_0^{\delta}|u'(\tau)|^2(1-h^{1/2}\tau)\,d\tau-|u(0)|^2\,.
\]
The  increasing sequence  of the eigenvalues of $\mathcal H_{h}$ (counting multiplicities) is denoted by
$(\lambda_n(\mathcal H_{h}))_{n\in \mathbb N}$. In \cite[Lem.~4.4\,\&\,Prop.~4.5]{HK-tams} it is proved that
\begin{equation}\label{eq:HK-tams}
\lambda_1(\mathcal H_h)=-1-h^{1/2}+o(h^{1/2}) \quad{\rm and}\quad 
\lambda_2(\mathcal H_{h})\geq \mathcal O(h^\rho)
\quad(h\to0_+)\,.
\end{equation}
We are going to refine the expansion of $\lambda_1(\mathcal H_h)$ by determining the term of order $h$.

\begin{lem}\label{lem:HK-tams}
Assume that $\rho\in(\frac14,\frac12)$. Then, as $h\to0_+$, 
\[
\lambda_1(\mathcal H_h)=-1-h^{1/2}-\frac12h+o(h)\,.
\]
\end{lem}
\begin{proof}
According to~\eqref{eq:HK-tams} there is a spectral gap of constant order, 
so it suffices to construct
a trial state that will give an energy estimate of sufficient accuracy.
To do this we expand the operator $\mathcal H_h$ formally in $h$ as
\[
\mathcal H_{h}=-\frac{d^2}{d\tau^2}+h^{1/2}\frac{d}{d\tau}+h\tau\frac{d}{d\tau}+\mathcal O(h^{\frac12+2\rho})\frac{d}{d\tau}\,,
\]
and note that, for $\rho\in(\frac14,\frac12)$, $h^{\frac12+2\rho}=o(h)$ as $h\to0_+$. 

We work on the half-line $\R_+$ and construct functions $u_0,u_1,u_2$ and 
coefficients $\mu_0,\mu_1,\mu_2$ such that
\begin{align*}
\Bigl(-\frac{d^2}{d\tau^2}-\mu_0\Bigr)u_0&=0,\\
\Bigl(-\frac{d^2}{d\tau^2}-\mu_0\Bigr)u_1&=-u_0'+\mu_1u_0,\\
\Bigl(-\frac{d^2}{d\tau^2}-\mu_0\Bigr)u_2&=-u_1'+\mu_1u_1-\tau
u_0'+\mu_2u_0,\\
\text{and}\quad u_i'(0)&=-u_i(0)~\text{for } i\in\{0,1,2\}\,.
\end{align*}
The natural choice is then to choose $u_0$ the eigenfunction in \eqref{eq:u0} and $\mu_0$ the corresponding eigenvalue. Then we choose $\mu_1$ so that $-u_0'+\mu_1u_0$ is orthogonal to $u_0$; after that we can determine $u_1$ since we can invert the operator $-\frac{d^2}{d\tau^2}-\mu_0$ on the orthogonal complement of $u_0$. Finally, we select $\mu_2$ so that $-u_1'+\mu_1u_1-\tau
u_0'+\mu_2u_0$ is othogonal to $u_0$ which allows us eventually to determine $u_2$. In that way we obtain
\begin{align*}
&\mu_0=-1\,,\quad u_0(\tau)=\sqrt{2}\,\exp(-\tau)\,,\\
&\mu_1=-1\,,\quad u_1(\tau)=0\,,\\
& \mu_2=-\frac12\,,\quad u_2(\tau)=\left(-\frac{d^2}{d\tau^2}+1\right)^{-1}\Bigl[\Bigl(\tau-\frac12\Bigr)u_0(\tau)\Bigr]=\Bigl(\frac{\tau^2}{4}-\frac{1}{8}\Bigr)u_0(\tau)\,.
\end{align*}
Now, consider the function,
\[
f(\tau)=\chi\left(\frac{\tau}{\delta}\right)\Big( u_0(\tau)+h^{1/2}u_1(\tau)+hu_2(\tau)\Big)\,,
\]
where $\chi\in C_c^\infty([0,\infty))$ satisfies,
\[
0\leq \chi\leq 1~{\rm in~}[0,\infty)\,,\quad \chi=1~{\rm
in~}[0,1/2)\quad\text{and}\quad\chi=0~\text{in~}[1/2,+\infty)\,.
\]
The function $f$ is in the domain of the operator $\mathcal H_{h}$, and by
construction it is almost normalized in the weighted Hilbert space, having
a norm of size $1+\mathcal O(h)$.
Moreover, a straight forward estimate shows that
\[
\|\{\mathcal H_{h}-(\mu_0+\mu_1h^{1/2}+\mu_2h
)\}f\|_{{L^2((0,\delta);(1-h^{1/2}\tau)\,d\tau)}}=o(h).
\]
The spectral theorem and \eqref{eq:HK-tams} now  completes the proof of Lemma~\ref{lem:HK-tams}.
\end{proof}

\subsection{Reducing the angular momentum}
We disqualify some values of the angular momentum $m$
from minimizing the right-hand side in \eqref{eq:ev=min}.

\begin{proposition}\label{prop:mfinite}
Assume that $b>0$, $\rho\in(\frac14,\frac12)$ and that $h\in(0,1)$. 
If $|m|>(1+\sqrt{2})b/2$, then
\[
\lambda_1(\mathcal H_{m,h}^{b,\rho})>\inf_{\ell\in\mathbb Z}\lambda_1(\mathcal H_{\ell,h}^{b,\rho}).
\]
\end{proposition}
\begin{proof}
We note that if $m=0$, then we have the bound
\[
\frac{h}{(1-h^{1/2}\tau)^{2}}\Bigl(m-\frac{b}{2}(1-h^{1/2}\tau)^2\Bigr)^2\leq \frac{b^2}{4}h
\]
on the potential term in $\mathcal H_{m,h}^{b,\rho}$. Comparing quadratic forms, 
\begin{equation}\label{eq:ev-m<ev-0}
\inf_{\ell\in\mathbb Z} \lambda_1\bigl(\mathcal H_{\ell,h}^{b,\rho}\bigr)\leq 
\lambda_1\bigl(\mathcal H_{0,h}^{b,\rho}\bigl)\leq \lambda_1(\mathcal H_h) +\frac{b^2}4h\,, 
\end{equation}
where $\lambda_1(\mathcal H_h)$ is the lowest eigenvalue of the operator introduced in \eqref{eq:Hh}.

We expand the square and estimate the potential term again, using the fact that $1-h^\rho<1-h^{1/2}\tau<1$,
\[
\begin{aligned}
\frac{1}{(1-h^{1/2}\tau)^{2}}\Bigl(m-\frac{b}{2}(1-h^{1/2}\tau)^2\Bigr)^2
&=
\frac{m^2}{(1-h^{1/2}\tau)^2}-mb+\frac{b^2}{4}(1-h^{1/2}\tau)^2\\
&\geq m^2-mb=\Bigl(m-\frac{b}{2}\Bigr)^2-\frac{b^2}{4}.
\end{aligned}
\]
We compare the quadratic forms and invoke~\eqref{eq:ev-m<ev-0} to find that
\[
\begin{aligned}
\lambda_1\bigl(\mathcal H_{m,h}^{b,\rho}\bigr)
&\geq \lambda_1(\mathcal H_h)+h\Bigl(\Bigl(m-\frac{b}{2}\Bigr)^2-\frac{b^2}{4}\Bigr)\\
&\geq \inf_{\ell\in\mathbb Z} \lambda_1\bigl(\mathcal H_{\ell,h}^{b,\rho}\bigr)
+h\Bigl(\Bigl(m-\frac{b}{2}\Bigr)^2-\frac{b^2}{2}\Bigr).
\end{aligned}
\]
If $|m|>(1+\sqrt{2})b/2$ then $(m-b/2)^2>b^2/2$ and thus
\[
\lambda_1\bigl(\mathcal H_{m,h}^b\bigr)
>\inf_{\ell\in\mathbb Z}
\lambda_1\bigl(\mathcal H_{\ell,h}^{b,\rho}\bigr).\qedhere
\]
\end{proof}

\subsection{A family of 1D operators}

Assume that $A>0$, $\rho\in(0,\frac12)$  and $b>0$ are fixed constants.  
Suppose that  the parameters $h\in(0,1)$ and $m\in\mathbb Z$ vary as follows
\begin{equation}\label{eq:cond-par}
h^{\frac12-\rho}<\frac13\quad{\rm and}\quad |m|\leq A\,.
\end{equation}
We introduce also the following ground state energy
\begin{equation}\label{eq:gs-m}
\hat\lambda(b,A)=\inf_{\substack{m\in\mathbb Z \\ |m|\leq A}} \lambda_1(\mathcal H_{m,h}^b)\,.
\end{equation}

\begin{proposition}\label{prop:w-op}
Given $A,b>0$ and $\rho\in(\frac14,\frac12)$, it holds that
\[
\hat\lambda(b,A) = -1+h^{1/2}+\Bigl(\hat\beta(b,A)-\frac12\Bigr)h+o(h)\,,
\]
where
\[
\hat\beta(b,A)=\inf_{\substack{m\in\mathbb Z\\ |m|\leq A}}\Bigl(m-\frac{b}2\Bigr)^2\,.
\]
\end{proposition}

\begin{rem}\label{rem:beta}
Given $b>0$, there exists  $A_0>0$ such that, for all $A\geq A_0$,
\[
\hat\beta(b,A)=\inf_{m\in\mathbb Z} \Bigl(m-\frac{b}2\Bigr)^2\,.
\]
\end{rem}

\begin{proof}
We will write estimates that hold uniformly with respect to $(m,h)$ 
obeying the conditions in \eqref{eq:cond-par}. A calculation shows that
\[
\begin{multlined}
\frac{1}{(1-h^{1/2}\tau)^2}\Bigl(m-\frac{b}{2}(1-h^{1/2}\tau)^2\Bigr)^2-\Bigl(m-\frac{b}2\Bigr)^2\\
=
(2-h^{1/2}\tau)\biggl(\frac{1}{(1-h^{1/2}\tau)^2}m^2 -\frac{b^2}{4}\biggr)h^{1/2}\tau.
\end{multlined}
\]
Thus, using that $0\leq h^{1/2}\tau \leq h^\rho$ for $\tau\in(0,\delta)$, we can write
\[
\Bigl|\frac{1}{(1-h^{1/2}\tau)^{2}}\Bigl(m-\frac{b}{2}(1-h^{1/2}\tau)^2\Bigr)^2-\Bigl(m-\frac{b}2\Bigr)^2\Bigr|\leq Ch^\rho\,,
\]
where $C$ is a constant independent from $m\in[-A,A]$.
Consequently, the min-max principle yields 
\[
\Bigl| \lambda_1(\mathcal H_{m,h}^{b,\rho})-\Bigl(\lambda_1(\mathcal H_h)+\Bigl(m-\frac{b}2\Bigr)^2h\Bigr)\Bigr|\leq Ch^{1+\rho}\,,
\]
where $\mathcal H_{h}$ is the operator introduced in \eqref{eq:Hh}. 
Now, using  Lemma~\ref{lem:HK-tams}, we finish the proof of Proposition~\ref{prop:w-op}.
\end{proof}

\subsection{End of proof}
We now have everything we need to finish the proof of Theorem~\ref{thm:KS}. 
Remember that $b>0$ was given in the theorem. We start by choosing some 
$\rho\in(\frac14,\frac12)$. Then, by combining
the Propositions~\ref{prop:mfinite} and~\ref{prop:w-op} (see also Remark~\ref{rem:beta})
with Lemma~\ref{lem:HK-tams} we find that, as $h\to 0_+$,
\[
\inf_{m\in\mathbb Z}\lambda_1(\mathcal H_{m,h}^{b,\rho})
=
-1-h^{1/2}+\Bigl(\inf_{m\in\mathbb Z}\Bigl(m-\frac{b}{2}\Bigr)-\frac{1}{2}\Bigr)h+o(h)
\,.\]
From~\eqref{eq:ev=min} we now conclude that, as $h\to 0_+$,
\[
\tilde\mu(h,b,\rho)=-h-h^{3/2}+\Bigl(\inf_{m\in\mathbb Z}\Bigl(m-\frac{b}{2}\Bigr)-\frac{1}{2}\Bigr)h^2+o(h^2).
\]
As we mentioned in the end of Subsection~\ref{sec:toring} this was sufficient to
prove Theorem~\ref{thm:KS'} which in turn was a reformulation of Theorem~\ref{thm:KS}.

\end{document}